\documentclass[a4paper]{amsart}

\usepackage[T1]{fontenc}
\usepackage[english]{babel}
\usepackage[alphabetic]{amsrefs}
\usepackage{amssymb}
\usepackage{geometry}
\usepackage{fullpage}
\usepackage[shortlabels]{enumitem}
\usepackage{graphicx}
\usepackage{hyperref}
\usepackage{mathrsfs}
\usepackage{mathtools}

\newtheorem{theorem}{Theorem}[section]
\newtheorem{lemma}[theorem]{Lemma}
\newtheorem{proposition}[theorem]{Proposition}
\newtheorem{corollary}[theorem]{Corollary}

\newtheorem{remark}[theorem]{Remark}
\numberwithin{equation}{section}

\def\N{\mathbb{N}}
\def\Z{\mathbb{Z}}
\def\C{\mathbb{C}}
\newcommand{\Q}{{\mathbb{Q}}}

\newcommand{\floor}[1]{\left\lfloor #1 \right\rfloor}

\allowdisplaybreaks[4]

\title{Schur positivity of nabla on Petrie symmetric functions}

\author{Menghao Qu}
\address{Scuola Normale Superiore\\
	Piazza dei Cavalieri 7,
	56126 Pisa\\ Italy}\email{menghao.qu@sns.it}

\begin{document}
\begin{abstract}
The Petrie symmetric function $G(k,n)$, introduced by Grinberg, is defined as the sum of monomial symmetric functions $m_\lambda$ indexed by partitions $\lambda\vdash n$ satisfying $\lambda_1<k$. This article demonstrates that the Schur positivity pattern of $\nabla^r G(k,n)$ for all $r\geq 1$ depends exclusively on whether $k$ divides $n$, thus answering an open problem of Bergeron noted in Grinberg's work.
\end{abstract}

\maketitle

\section{Introduction}
The theory of symmetric functions is fundamental to several major areas of mathematics, most notably algebraic combinatorics and representation theory. A symmetric function is termed Schur positive if its expansion coefficients in the Schur basis belong to $\N[q,t]$, the set of polynomials in $q$ and $t$ with nonnegative integer coefficients. This property provides a fundamental link between the theory of symmetric functions and the representation theory of the symmetric group $\mathfrak{S}_n$. In many important cases, such positivity reflects the existence of a bigraded $\mathfrak{S}_n$-module whose bigraded Frobenius characteristic is given by the symmetric function in question. Two prominent realizations of this correspondence are the modified Macdonald polynomials \cites{MR2138143} (associated with the Garsia--Haiman module \cites{MR1214091,MR1839919}) and $\nabla e_n$ (associated with the diagonal harmonics \cites{MR1256101,MR1918676}). The nabla operator $\nabla$, introduced by Bergeron and Garsia \cites{MR1726826} and defined by its diagonal action on the modified Macdonald basis, has profoundly shaped modern algebraic combinatorics.

A celebrated result in this area is the shuffle theorem, which expresses $\nabla e_n$ as a combinatorial sum over parking functions, as originally conjectured in \cite{MR2115257}. Even before Carlsson and Mellit \cite{MR3787405} provided a formal proof, this conjecture inspired significant generalizations and refinements \cites{MR1803316,MR2255191,MR2418288,MR2957232,MR3556418}. Among these, Haglund, Morse, and Zabrocki \cite{MR2957232} introduced the creation operators $\C_a$ for $a\in\Z$. For any composition $\alpha=(\alpha_1,\dots,\alpha_{\ell})\vDash n$, the successive application of these operators to the constant $1$ yields the symmetric function $\C_{\alpha}=\C_{\alpha}[X;q]:=\C_{\alpha_1}\cdots\C_{\alpha_{\ell}}(1)$, a fundamental building block in the theory. The algebraic decomposition $e_n = \sum_{\alpha\vDash n} \C_{\alpha}$ gave rise to the compositional shuffle conjecture, which explicitly formulates the action of $\nabla$ on $\C_{\alpha}$. In particular, $\nabla\C_\alpha)$ is Schur positive. This follows from the compositional shuffle theorem, which expresses the combinatorial side of $\nabla\C_\alpha$ as a positively weighted sum of vertical LLT polynomials \cite{MR1434225}, whose Schur positivity was established by Grojnowski and Haiman \cite{grojnowski2007affine}. Inspired by these foundational developments, the study of Schur positivity for the nabla operator acting on various families of symmetric functions remains an active research direction.

In a seminal article \cite{MR1803316}, a series of conjectures was presented concerning the Schur positivity of $\nabla$ acting on several foundational bases, namely Schur functions, monomial symmetric functions, and modified Hall--Littlewood polynomials. To the best of our knowledge, a proof of the Schur positivity of $\nabla s_{\lambda}$, up to a sign, is still lacking, although its monomial expansion has been independently obtained through various methods \cites{MR4930329,extendingsf}. Regarding the modified Hall--Littlewood polynomials, we recently settled two related conjectures for the case of two-column partitions \cite{qu26} by establishing identities that express them in terms of known Schur-positive symmetric functions. Turning to monomial symmetric functions, one might expect $m_{\mu}$ to admit a $\C_{\alpha}$-expansion, drawing inspiration from the Schur positivity of $\nabla \C_{\alpha}$. Sergel \cite{MR3940644} demonstrated that $m_{n,1^k}$ admits an expansion in $\C_{\alpha}$ with coefficients in $\N[q]$, successfully resolving the case of hook partitions. In a previous joint work with Xin \cite{MR4921577}, we established a recurrence relation for the two-column case utilizing the $\C_a$ operators, which yielded a similar expansion in the $\C_{\alpha}$ with coefficients in $\N[q]$. It is important to note that the set $\{\C_{\alpha} : \alpha \vDash n\}$ is not a basis for $\Lambda^{(n)}$. Because of this, systematically extending these expansions to general cases poses a significant challenge. Recently, Qiu and Zhang \cite{qiuzhang26} extended this recursive approach to arbitrary partitions. Their approach relies on the fact that the Loehr--Warrington conjecture \cite{MR2418288} was proved by Blasiak, Haiman, Morse, Pun, and Seelinger \cite{MR4930329}, who provided an LLT expansion of $\nabla^r s_\lambda$ with integral coefficients. By combining their work with both the Schur positivity of LLT polynomials due to Grojnowski and Haiman \cite{grojnowski2007affine} and the integrality of the inverse Kostka matrix established by Eğecioğlu and Remmel \cite{MR1034417}, Qiu and Zhang proved the Schur positivity of $(-1)^{|\mu|-\ell(\mu)}\nabla^r m_{\mu}$.

The Petrie symmetric function $G(k,n)$, introduced by Grinberg~\cite{MR4511157} and also appearing in \cite{MR4388836} under the name truncated homogeneous symmetric function, is defined by
\begin{equation}
G(k,n)=\sum_{\mu\vdash n, \mu_1<k} m_\mu.
\end{equation}
In particular, we have $G(2,n)=e_n$, and $G(k,n)=h_n$ whenever $k>n$. This family of symmetric functions has garnered considerable interest due to the striking properties of its Schur expansions \cites{MR4169831,MR4556182}. Naturally, this prompts an investigation into the Schur positivity of its image under the $\nabla$ operator. The classification of this exact positivity pattern was posed as an open problem by Bergeron and recorded in \cite{MR4511157}. Since signed Schur positivity for $\nabla^r m_\mu$ already involves the nontrivial sign $(-1)^{|\mu|-\ell(\mu)}$, the answer for Petrie symmetric functions is far from obvious.

In a previous joint work with Xin \cite{MR4921577}, we established a symmetric function identity relating $m_{k^a}$, $G(k,ak)$, and the $\C_{a}$ operators, which prompted us to conjecture the Schur positivity of $(-1)^{ak}\nabla G(k,ak)$. Building upon this identity and the recent recursive framework introduced by Qiu and Zhang \cite{qiuzhang26}, we resolve this conjecture. Furthermore, for the case where $k \nmid n$, we derive analogous identities that establish additional Schur positivity results. Specifically, we prove the following:

\begin{theorem}\label{thm-k-mid-n}
For integers $k\geq 2$, $a\geq 1$, and $r\geq 1$, $(-1)^{ak}\nabla^{r} G(k,ak)$ is Schur positive.
\end{theorem}

\begin{theorem}\label{thm-k-nmid-n}
For integers $k\geq 2$, $n\geq 1$, and $r\geq 1$ with $k\nmid n$, $(-1)^{n-1}\nabla^{r} G(k,n)$ is Schur positive. 
\end{theorem}

Together, these two theorems completely determine the sign pattern for Schur positivity of $\nabla^{r} G(k,n)$, thereby resolving the open problem posed in \cite{MR4511157}*{Conjecture 5.2}. The remainder of this paper is organized as follows. Section 2 provides the requisite background on symmetric functions and reviews key components of the Qiu--Zhang recursion. In Section 3, the proof is divided into two cases according to whether $k$ divides $n$. We first handle the divisible case $n=ak$, combining a symmetric function identity from \cite{MR4921577} with a specialized form of the Qiu--Zhang recursion to obtain a nonnegative rational $\C_{\alpha}$-expansion of $(-1)^{ak}G(k,ak)$. We then treat the remaining case $k\nmid n$ by expressing $G(k,n)$ in terms of the $\C_a$-operators applied to signed monomial symmetric functions indexed by rectangular partitions. The Schur positivity of $(-1)^{n-1}\nabla^r G(k,n)$ in the nondivisible case then follows from the Schur positivity of $\nabla^r$ acting on these functions.

\section{Preliminary}

In this section, we review the essential definitions and fix the notation used throughout the paper. While we only briefly recall the standard definitions concerning symmetric functions and Macdonald polynomials, comprehensive treatments can be found in \cites{MR1354144, MR1676282, MR2371044, MR2538310}. In addition, we state several key results from the recent work of \cite{qiuzhang26}, which will be crucial for our subsequent proofs.

\subsection{Symmetric functions}
A \emph{partition} $\lambda$ is a finite nonincreasing sequence $\lambda_{1}\geq\lambda_{2}\geq\cdots \geq\lambda_{\ell}>0$ of positive integers. Each $\lambda_{i}$ is called the $i$-th \emph{part} of $\lambda$. The number of parts, denoted by $\ell(\lambda)$, is called the \emph{length} of $\lambda$, and the sum of its parts, $|\lambda|:=\sum_{i=1}^{\ell}\lambda_{i}$, is called the \emph{size} of $\lambda$. We write $\lambda\vdash n$ to denote that $\lambda$ is a partition of size $n$.

The \emph{Young diagram} of $\lambda$ is an array of unit squares, called
cells, with $\lambda_{i}$ cells in the $i$th row (from the bottom), with the first cell in each row left-justified. The \emph{conjugate partition}, $\lambda'$, is the partition whose Young diagram is obtained from $\lambda$ by reflecting across the diagonal $y=x$.

Similarly, a \emph{composition} $\alpha=(\alpha_{1},\alpha_{2},\cdots,\alpha_{\ell})$ of an integer $n$ is an ordered sequence of parts $\alpha_{i}>0$ that sum up to $n$. We write $\alpha\vDash n$ when $\alpha$ is a composition of $n$. The length $\ell(\alpha)$ of a composition $\alpha$ is the number of its parts.

Let $\Lambda$ denote the ring of symmetric functions over $\Q(q,t)$. It is equipped with a natural grading $\Lambda = \bigoplus_{n \geq 0} \Lambda^{(n)}$, where $\Lambda^{(n)}$ consists of homogeneous symmetric functions of degree $n$, defined by assigning $\deg(p_k) = k$. The \emph{monomial} $\{m_{\lambda}\}_{\lambda\vdash n}$, \emph{elementary} $\{e_{\lambda}\}_{\lambda\vdash n}$, \emph{complete homogeneous} $\{h_{\lambda}\}_{\lambda\vdash n}$, \emph{power sum} $\{p_{\lambda}\}_{\lambda\vdash n}$, and \emph{Schur} $\{s_{\lambda}\}_{\lambda\vdash n}$ functions constitute five classical bases for the space $\Lambda^{(n)}$. Throughout, we adopt the convention that $h_r=0$ for $r<0$ and $h_0=1$. 

The fundamental \emph{involution} $\omega:\Lambda \rightarrow \Lambda$ acts on the Schur basis via
\begin{equation*}
\omega(s_{\lambda})=s_{\lambda'}    
\end{equation*}
Furthermore, the Schur functions form an orthonormal basis with respect to the \emph{Hall scalar product}, satisfying $\langle s_{\lambda},s_{\mu}\rangle=\delta_{\lambda,\mu}$, where $\delta_{\lambda,\mu}$ denotes the Kronecker delta. We will adopt the notation $F[X] := F(x_1, x_2, \dots)$ to denote a symmetric function evaluated on the variable set $X$.

For $\mu\vdash n$, we denote by
\begin{equation*}
\tilde{H}_{\mu}[X;q,t]=\sum\limits_{\lambda\vdash n}\tilde{K}_{\lambda,\mu}(q,t)s_{\lambda},
\end{equation*}
the \emph{modified Macdonald polynomial}, where the coefficients
\begin{equation*}
\tilde{K}_{\lambda,\mu}(q,t)=t^{n(\mu)}K_{\lambda,\mu}(q,t^{-1})\in\N[q,t],
\end{equation*}
are the \emph{modified $q,t$-Kostka polynomials}. The family $\{\tilde{H}_{\mu}[X;q,t]\}_{\mu\vdash n}$ also forms a basis for $\Lambda^{(n)}$. 

Let $\nabla$ be the linear operator on symmetric functions which satisfies
\begin{equation*}
\nabla \tilde{H}_{\mu}[X;q,t]=T_{\mu}\tilde{H}_{\mu}[X;q,t],    
\end{equation*}
where $T_{\mu}=q^{n(\mu')}t^{n(\mu)}$ with $n(\mu)=\sum_{i}(i-1)\mu_i$.

We will extensively use plethystic calculus throughout this work. Let $E(t_{1},t_{2},t_{3},\cdots)$ be a formal series of rational functions in the parameters $t_{1},t_{2},t_{3},\cdots$ Define the \emph{plethystic substitution} of $E$ into $p_{k}$, denoted $p_{k}[E]$, by
\begin{equation*}
p_{k}[E]=E(t_{1}^{k},t_{2}^{k},\cdots).
\end{equation*}
For $f=\sum_{\lambda}c_{\lambda}(q,t)p_{\lambda}\in \Lambda$, we have $f[E]=\sum_{\lambda}c_{\lambda}(q,t)\prod_{k=1}^{\ell(\lambda)}p_{k}[E]$. Plethystic notation is a powerful tool that simplifies many calculations in the theory of symmetric functions and it is highly amenable to implementation in symbolic computation software such as Maple, Mathematica, and SageMath. Readers unfamiliar with plethystic notation can refer to \cite{MR2371044}*{Chapter 1}. For convenience, we collect below several identities, stated without proof, that will be used extensively in the subsequent sections.

\begin{lemma}[addition formulas]\label{lem-addition-formulas}
\begin{align}
e_{n}[X+Y]&=\sum\limits_{i=0}^{n}e_{i}[X]e_{n-i}[Y],\quad e_{n}[X-Y]=\sum\limits_{i=0}^{n}e_{i}[X]e_{n-i}[-Y],\\
h_{n}[X+Y]&=\sum\limits_{i=0}^{n}h_{i}[X]h_{n-i}[Y],\quad h_{n}[X-Y]=\sum\limits_{i=0}^{n}h_{i}[X]h_{n-i}[-Y].
\end{align}
\end{lemma}

\begin{lemma}\label{lem-plethystic-minus}
For $f[X]\in \Lambda^{(n)}$,
\begin{equation}
f[-X]=(-1)^{n}\omega f[X].
\end{equation}
\end{lemma}

\begin{lemma}\label{lem-plethystic-h}
If $u$ and $v$ are monomials, then we have
\begin{equation}
h_{n}[(1-u)v]=\begin{cases}
1, & \text{ if } n=0,\\
(1-u)v^{n}, & \text{ if } n\geq 1.
\end{cases}
\end{equation}
\end{lemma}

\subsection{Creation operators \texorpdfstring{$\C_a$}{C\_a}}
In \cite{MR2957232}, the authors introduced a family of creation operators $\C_a$ for $a\in\Z$. For any $f[X]\in \Lambda$ and $a\in \Z$, the action of $\C_a$ is defined by
\begin{equation}
\C_{a}f[X]=\left(-\frac{1}{q}\right)^{a-1}f\left[X-\frac{1-\frac{1}{q}}{z}\right]\sum_{n\geq 0}z^{n}h_{n}[X]\Bigg|_{z^a}.    
\end{equation}

For any $\alpha=(\alpha_{1},\alpha_{2},\cdots,\alpha_{\ell})\vDash n$, let 
\begin{equation}
\C_{\alpha}:=\C_{\alpha}[X;q]=\C_{\alpha_{1}}\C_{\alpha_{2}}\cdots \C_{\alpha_{\ell}}(1).
\end{equation}
Because it arises from a transformation of the Hall--Littlewood polynomials, the set $\{\C_{\lambda}:\lambda\vdash n\}$ forms a basis of $\Lambda^{(n)}$. In contrast, $\{\C_{\alpha}:\alpha\vDash n\}$ does not.

They also proved that $e_n=\sum_{\alpha\vDash n}\C_{\alpha}$. This yields a powerful framework for analyzing symmetric functions by reducing them to smaller, more manageable building blocks. A prime example of this is found in \cite{MR3787405}, where Carlsson and Mellit proved the compositional shuffle conjecture, thereby establishing the original shuffle theorem as a corollary. Later, Mellit \cite{MR4348234} successfully resolved the compositional $(km,kn)$-shuffle conjecture proposed in \cite{MR3556418}. Building upon the combinatorial interpretation of $\nabla^{r} \C_{\alpha}$ and the known properties of vertical LLT polynomials \cites{grojnowski2007affine}, one can deduce the following properties:

\begin{proposition}\label{prop-nabla-r-C-alpha}
For integer $r\geq 1$, $\alpha\vDash n$ and $\lambda\vdash n$, 
\begin{equation}
\langle \nabla^r \C_{\alpha},s_{\lambda}\rangle\in \N[q,t].
\end{equation}
\end{proposition}

\subsection{Qiu-Zhang's recursion}

Let us recall some notation from \cite{qiuzhang26}. A \emph{labeled multiset} $A$ is a finite multiset of positive integers in which identical elements are distinguished by subscripts. Specifically, if a value appears multiple times, its copies are labeled with subscripts $1, 2, \dots$; if it appears only once, it is written without a subscript. Since its labeled elements are pairwise distinct, a labeled multiset can be treated as a finite set: subsets, unions, differences, and set partitions of labeled multisets are understood in the ordinary sense. 

In algebraic expressions, the elements of a labeled multiset $A$ are evaluated strictly by their underlying integer values. Under this convention, the sum and length of $A$ are given by
\begin{equation}
S_{A}:=\sum_{x\in A}x \quad \text{ and }\quad \ell(A):=|A|.
\end{equation}
where $S_\varnothing=0$. By dropping the labels and arranging the values of $A$ in weakly decreasing order, one obtains a partition denoted by $\mu(A)$. To each labeled multiset $A$, one associates the signed monomial symmetric function:
\begin{equation}
F_A := (-1)^{|\mu(A)|-\ell(\mu(A))} m_{\mu(A)}=(-1)^{S_A-\ell(A)} m_{\mu(A)}.
\end{equation}
By convention, we set $F_{\varnothing}:=1$. Finally, a scaled variant of this function is given by
\begin{equation}
\tilde{F}_A := \left(\prod_{i\geq 1} m_i(A)!\right)F_A,    
\end{equation}
where $m_i(A)$ represents the multiplicity of the value $i$ within $A$ (which is equivalent to $m_i(\mu(A))$).

For example, let $A=\{1,2_1,2_2,2_3,3_1,3_2\}$. Then we yield $S_{A}=13$, $\ell(A)=6$, $\mu(A)=(3,3,2,2,2,1)$, $F_{A}=-m_{3,3,2,2,2,1}$, and $\tilde{F}_{A}=-12m_{3,3,2,2,2,1}$.

Qiu and Zhang proved the following recursion. 

\begin{lemma}\cite{qiuzhang26}*{Theorem 4.1}
Let $A$ be a labeled multiset of positive integers, and let $c\geq 1$ be an integer such that $x\geq c$ for all $x\in A$. Then
\begin{equation}\label{equ-qiu-zhang-recursion}
\tilde{F}_{A\cup\{c\}}-\sum_{j=1}^{c-1}q^{j-1}\C_j(\tilde{F}_{A\cup\{c-j\}})=\sum_{T\subseteq A}\left(\sum_{b=1}^{S_T}\alpha_{T,b}\,\C_{c+S_T-b}(\tilde{F}_{(A\backslash T)\cup\{b\}})+\beta_T\,\C_{c+S_T}(\tilde{F}_{A\backslash T})\right),
\end{equation}
where $\alpha_{T,b},\beta_T \in \N[q]$ are defined as follows. Set $\alpha_{\varnothing,b}=0$, and for $T\neq\varnothing$, define
\begin{equation}
\alpha_{T,b}=(|T|-1)!\sum_{x\in T, b\leq x}q^{(c+x-b-1) \bmod x},
\end{equation}
where $y \bmod x$ denotes the least nonnegative residue of $y$ modulo $x$. Furthermore, $\beta_T$ is defined by
\begin{equation}
\beta_T=(|T|+1)![c]_q+|T|!\sum_{x\in T}\left(q^c+q^{c+1}+\cdots+q^{x-1}\right),
\end{equation}
with $[c]_q := 1+q+\dots+q^{c-1}$.
\end{lemma}

From this recursion, they derived the following results:

\begin{lemma}\cite{qiuzhang26}*{Theorem 4.2}\label{lem-qiu-zhang-tilde-F-C-expansion}
For any labeled multiset $A$ of positive integers,
\begin{equation}
\tilde{F}_{A}\in \sum_{\alpha\vDash S_{A}}\N[q]\C_{\alpha}.
\end{equation}
\end{lemma}

\begin{lemma}\cite{qiuzhang26}\label{lem-qiu-zhang-rational-C-expansion}
For any partition $\mu$,
\begin{equation}
(-1)^{|\mu|-\ell(\mu)}m_{\mu}\in \sum_{\alpha\vDash |\mu|}\Q_{\geq 0}[q]\C_{\alpha}.
\end{equation}
\end{lemma}

\begin{lemma}\cite{qiuzhang26}\label{lem-qiu-zhang-integrality}
For any partition $\mu$ and integer $r\geq 1$,
\begin{equation}
\langle \nabla^r m_{\mu},s_{\lambda}\rangle\in \Z[q,t].
\end{equation}
\end{lemma}

By combining Proposition \ref{prop-nabla-r-C-alpha} with Lemmas \ref{lem-qiu-zhang-rational-C-expansion} and \ref{lem-qiu-zhang-integrality}, they obtain the Schur positivity of $\nabla$ on signed monomial symmetric functions.

\begin{corollary}\cite{qiuzhang26}*{Theorem 1.2}
For any partition $\mu$ and integer $r\geq 1$,
\begin{equation}
\langle (-1)^{|\mu|-\ell(\mu)}\nabla^r m_{\mu},s_{\lambda}\rangle\in \N[q,t].
\end{equation}
\end{corollary}

\section{Proof of the main results}

Since each Petrie symmetric function is a sum of monomial symmetric functions, Lemma \ref{lem-qiu-zhang-integrality} implies that the Schur coefficients of $\nabla^{r}G(k,n)$ satisfy the same integrality property, that is,
\begin{equation*}
\langle \nabla^r G(k,n),s_{\lambda}\rangle\in \Z[q,t].
\end{equation*}
Thus, to establish Schur positivity, it remains to show that, up to a sign, the original functions $G(k,n)$ can be expressed as linear combinations of $\C_{\alpha}$ with coefficients in $\Q_{\geq 0}[q]$.

\subsection{Proof of Theorem \ref{thm-k-mid-n}}
In joint work with Xin, we established the following identity and conjectured the Schur positivity of $(-1)^{ak}\nabla G(k,ak)$.

\begin{lemma}\cite{MR4921577}*{Theorem 3.7}
For integers $k\geq 2$ and $a\geq 1$, we have
\begin{equation}
(-1)^{ak}G(k,ak)=(-1)^{(k-1)a}m_{k^a}-q^{k-1}\sum_{i=1}^a\C_{ki}\left((-1)^{(k-1)(a-i)}m_{k^{a-i}}\right).
\end{equation}
\end{lemma}

\begin{proposition}
For integers $k\geq 2$ and $a\geq 1$, there exists an element $R_{k,a}\in\Lambda$, which is a linear combination of the $\C_{\alpha}$ with coefficients in $\Q_{\geq 0}[q]$, such that
\begin{equation}\label{equ-mka-c-R-k-a}
(-1)^{(k-1)a}m_{k^a}=\sum_{i=1}^a\left(\frac{i}{a}[k]_q+\frac{a-i}{a}q^{k-1}\right)\C_{ki}\left((-1)^{(k-1)(a-i)}m_{k^{a-i}}\right)+R_{k,a}.
\end{equation}
\end{proposition}
\begin{proof}

Consider the labeled multiset $A=\{k_1,k_2,\ldots,k_{a-1}\}$ consisting of $a-1$ labeled copies of $k$. We apply \eqref{equ-qiu-zhang-recursion} with $c=k$. The condition $x\geq c$ is naturally satisfied since all elements in $A$ are equal to $k$, which yields
\begin{equation*}
\tilde{F}_{A\cup\{k\}}=\sum_{j=1}^{k-1}q^{j-1}\C_j\left(\tilde{F}_{A\cup\{k-j\}}\right)+\sum_{T\subseteq A}\left(\sum_{b=1}^{S_T}\alpha_{T,b}\,\C_{k+S_T-b}
\left(\tilde{F}_{(A\setminus T)\cup\{b\}}\right)+\beta_T\,\C_{k+S_T}\left(\tilde{F}_{A\setminus T}\right)
\right).
\end{equation*}
By definition, since $A\cup\{k\}=\{k_1,k_2,\ldots,k_a\}$, we obtain $\tilde F_{A\cup\{k\}}=a!(-1)^{(k-1)a}m_{k^a}$. Dividing the above identity by $a!$, we collect the terms of the form $\C_{ki}\left((-1)^{(k-1)(a-i)}m_{k^{a-i}}\right)$ for $1\leq i\leq a$.

First, consider the $\beta_T$-terms. Fix $T\subseteq A$ and let $s:=|T|$. Then $S_T=ks$. Since every element $x\in T$ is equal to $k$, the second sum in the definition of $\beta_T$ is empty, meaning $q^k+q^{k+1}+\cdots+q^{k-1}=0$. Hence, we have $\beta_T=(s+1)![k]_q$. Evaluating this yields
\begin{align*}
\frac{1}{a!}\sum_{T\subseteq A}\beta_T\,\C_{k+S_T}\left(\tilde F_{A\setminus T}\right)&=\frac{1}{a!}\sum_{s=0}^{a-1}\sum_{T\subseteq A:|T|=s}\beta_T\,\C_{k+S_T}\left(\tilde{F}_{A\setminus T}\right)\\
&=\frac{1}{a!}\sum_{s=0}^{a-1}\binom{a-1}{s}(s+1)![k]_q\C_{k(s+1)}\left((a-1-s)!
(-1)^{(k-1)(a-1-s)}m_{k^{a-1-s}}\right)\\
&=\sum_{s=0}^{a-1}\frac{s+1}{a}[k]_q\C_{k(s+1)}\left((-1)^{(k-1)(a-1-s)}
m_{k^{a-1-s}}\right)\\
&=\sum_{i=1}^{a}\frac{i}{a}[k]_q\C_{ki}\left((-1)^{(k-1)(a-i)}m_{k^{a-i}}\right).
\end{align*}

Next, consider the $\alpha_{T,b}$-terms with $b=k$. If $T=\varnothing$, then $\alpha_{T,k}=0$. Assume $|T|=s\geq 1$. Since every $x\in T$ equals $k$, we have
\begin{align*}
\alpha_{T,k}=(s-1)!\sum_{x\in T}q^{(k+x-k-1)\bmod x}=(s-1)!\sum_{x\in T}q^{(x-1)\bmod x}=s!q^{k-1}.
\end{align*}
The sum of the corresponding terms is
\begin{align*}
\frac{1}{a!}\sum_{T\subseteq A}\alpha_{T,k}\C_{k+S_{T}-k}\left(\tilde{F}_{(A\backslash T)\cup k}\right)&=\frac{1}{a!}\sum_{s=1}^{a-1}\sum_{T\subseteq A: |T|=s}\alpha_{T,k}\C_{k+S_{T}-k}\left(\tilde{F}_{(A\backslash T)\cup k}\right)\\
&=\frac{1}{a!}\sum_{s=1}^{a-1}\binom{a-1}{s}s!q^{k-1}\C_{ks}\left((a-s)!(-1)^{(k-1)(a-s)}m_{k^{a-s}}\right)\\
&=\sum_{s=1}^{a-1}\frac{a-s}{a}q^{k-1}\C_{ks}\left((-1)^{(k-1)(a-s)}m_{k^{a-s}}\right)\\
&=\sum_{i=1}^{a}\frac{a-i}{a}q^{k-1}\C_{ki}\left((-1)^{(k-1)(a-i)}m_{k^{a-i}}\right).
\end{align*}

Indeed, every remaining term is of the form
$c(q)\C_d(\tilde{F}_B)$, where $c(q)\in \N[q]$ and $d\geq 0$. By Lemma \ref{lem-qiu-zhang-tilde-F-C-expansion}, $\tilde{F}_B\in \sum_\alpha \N[q]\C_\alpha$. After division by $a!$, the coefficients lie in $\Q_{\geq 0}[q]$. This completes the proof.
\end{proof}

\begin{corollary}
For integers $k\geq 2$, $a\geq 1$, and $r\geq 1$, we have
\begin{equation}
\langle (-1)^{ak}\nabla^r G(k,ak),s_{\lambda}\rangle\in \N[q,t].
\end{equation}
\end{corollary}
\begin{proof}
Using the identity $[k]_q=[k-1]_q+q^{k-1}$, we obtain
\begin{align*}
\frac{i}{a}[k]_q+\frac{a-i}{a}q^{k-1}=\frac{i}{a}([k-1]_q+q^{k-1})+\frac{a-i}{a}q^{k-1}=q^{k-1}+\frac{i}{a}[k-1]_q.
\end{align*}
Substituting this into \eqref{equ-mka-c-R-k-a} yields
\begin{align*}
(-1)^{ak}G(k,ak)=\sum_{i=1}^{a}\frac{i}{a}[k-1]_q\C_{ki}\left((-1)^{(k-1)(a-i)}m_{k^{a-i}}\right)+R_{k,a}\in\sum_{\alpha\vDash ak}\Q_{\geq 0}[q]\C_{\alpha}.
\end{align*}
This completes the proof.
\end{proof}

\subsection{Proof of Theorem \ref{thm-k-nmid-n}}
For the case $k \nmid n$, we express $G(k,n)$ in terms of the $\C_a$-operators applied to monomial symmetric functions indexed by rectangular partitions. We begin with the following lemma, whose proof uses calculations similar to those in \cite{MR4921577}.

\begin{lemma}\label{lem-crmka}
For integers $k\geq 2$, $a\geq 0$, and $r\in \Z$, we have
\begin{align}
\C_{r}m_{k^a}=\left(-\frac{1}{q}\right)^{r-1}\left(m_{k^a}h_r+\sum_{i=0}^{a-1}(-1)^{a-i}\frac{q^k-1}{q^k}m_{k^i}h_{k(a-i)+r}\right).
\end{align}
\end{lemma}
\begin{proof}
For $a=0$, we have, by definition, $\C_{r}(1)=(-1/q)^{r-1}h_r$. Now assume that $a\geq 1$, using the definition of the $\C_r$-operators together with Lemmas \ref{lem-addition-formulas}, \ref{lem-plethystic-minus}, and \ref{lem-plethystic-h}, we obtain
\begin{align*}
\C_{r}m_{k^a}&=\C_{r}e_{a}[p_k[X]]\\
&=\left(-\frac{1}{q}\right)^{r-1}e_{a}\left[p_k[X]-\frac{1-\frac{1}{q^k}}{z^k}\right]\sum_{n\geq 0}z^n h_{n} \Big |_{z^{r}}\\
&=\left(-\frac{1}{q}\right)^{r-1}\sum_{i=0}^{a}e_{i}[p_k[X]]e_{a-i}\left[-\frac{1-\frac{1}{q^k}}{z^k}\right]\sum_{n\geq 0}z^n h_{n}\Big |_{z^{r}}\\
&=\left(-\frac{1}{q}\right)^{r-1}\sum_{i=0}^{a}m_{k^i}(-1)^{a-i}h_{a-i}\left[\frac{1-\frac{1}{q^k}}{z^k}\right]\sum_{n\geq 0}z^n h_{n}\Big |_{z^{r}}\\
&=\left(-\frac{1}{q}\right)^{r-1}\sum_{i=0}^{a}(-1)^{a-i}m_{k^i}h_{a-i}\left[\frac{q^k-1}{q^k}\right]\frac{1}{z^{k(a-i)}}\sum_{n\geq 0}z^n h_{n}\Big |_{z^{r}}\\
&=\left(-\frac{1}{q}\right)^{r-1}\left(m_{k^a}h_r+\sum_{i=0}^{a-1}(-1)^{a-i}\frac{q^k-1}{q^k}m_{k^i}h_{k(a-i)+r}\right).
\end{align*}
This completes the proof.
\end{proof}

Grinberg proved the following symmetric function identity.

\begin{lemma}\cite{MR4511157}*{Theorem 2.19}
For integers $k\geq 2$ and $n\geq 1$, 
\begin{equation}\label{equ-G-mh}
G(k,n)=\sum\limits_{i=0}^{\floor{n/k}}(-1)^{i}m_{k^{i}}h_{n-ki},    
\end{equation}
where $\floor{n/k}$ denotes the greatest integer less than or equal to $n/k$.
\end{lemma}

Combining the two identities above, we obtain the following result.

\begin{theorem}\label{thm-gkncexpansion}
For integers $k\geq 2$ and $n\geq 1$, 
\begin{equation}\label{equ-Gkn-C-m}
(-1)^{n-1}G(k,n)=q^{n-k\floor{n/k}-1}\sum_{i=0}^{\floor{n/k}}\C_{n-ki}\left((-1)^{(k-1)i}m_{k^i}\right).
\end{equation}
\end{theorem}

\begin{proof}
By Lemma \ref{lem-crmka}, we have
\begin{equation*}
\C_{n-ki}m_{k^i}=\left(-\frac{1}{q}\right)^{n-ki-1}\left(m_{k^i}h_{n-ki}+\sum_{j=0}^{i-1}(-1)^{i-j}\frac{q^k-1}{q^k}m_{k^j}h_{n-kj}\right).
\end{equation*}
Substituting this into the right-hand side of the desired identity and expanding, we obtain
\begin{align*}
\mathrm{RHS} &= q^{n-k\floor{n/k}-1}\sum_{i=0}^{\floor{n/k}}(-1)^{(k-1)i}\left(-\frac{1}{q}\right)^{n-ki-1}\left(m_{k^i}h_{n-ki}+\sum_{j=0}^{i-1}(-1)^{i-j}\frac{q^k-1}{q^k}m_{k^j}h_{n-kj}\right)\\ 
&=\sum_{i=0}^{\floor{n/k}}(-1)^{n-i-1}q^{k(i-\floor{n/k})}m_{k^i}h_{n-ki} + \sum_{i=1}^{\floor{n/k}}\sum_{j=0}^{i-1}(-1)^{n-j-1}q^{k(i-\floor{n/k})}(1-q^{-k})m_{k^j}h_{n-kj}.
\end{align*}

For the second term, we exchange the order of summation:
\begin{align*}
\sum_{i=1}^{\floor{n/k}}\sum_{j=0}^{i-1}(-1)^{n-j-1}&q^{k(i-\floor{n/k})}(1-q^{-k})m_{k^j}h_{n-kj}\\
&=\sum_{j=0}^{\floor{n/k}-1}\sum_{i=j+1}^{\floor{n/k}}(-1)^{n-j-1}q^{k(i-\floor{n/k})}(1-q^{-k})m_{k^j}h_{n-kj}\\
&=\sum_{j=0}^{\floor{n/k}-1}(-1)^{n-j-1}m_{k^j}h_{n-kj}\sum_{i=j+1}^{\floor{n/k}}\left(q^{k(i-\floor{n/k})}-q^{k(i-1-\floor{n/k})}\right)\\ 
&=\sum_{j=0}^{\floor{n/k}-1}(-1)^{n-j-1}(1-q^{k(j-\floor{n/k})})m_{k^j}h_{n-kj}.
\end{align*}
Adding this to the first term gives
\begin{align*}
\mathrm{RHS}&=\sum_{i=0}^{\floor{n/k}}(-1)^{n-i-1}q^{k(i-\floor{n/k})}m_{k^i}h_{n-ki}+ \sum_{j=0}^{\floor{n/k}-1}(-1)^{n-j-1}(1-q^{k(j-\floor{n/k})})m_{k^j}h_{n-kj} \\
&=\sum_{i=0}^{\floor{n/k}}(-1)^{n-i-1}m_{k^i}h_{n-ki}.
\end{align*}

On the other hand, by \eqref{equ-G-mh}, the left-hand side is
\begin{equation*}
\mathrm{LHS}=\sum_{i=0}^{\floor{n/k}}(-1)^{n-i-1}m_{k^i}h_{n-ki}.
\end{equation*}
Thus $\mathrm{LHS}=\mathrm{RHS}$, completing the proof.
\end{proof}

For $k\nmid n$, we have $n-ki>0$ for all $0\leq i\leq \floor{n/k}$. Applying $\nabla^r$ and Proposition \ref{prop-nabla-r-C-alpha} gives Schur coefficients in $\Q_{\ge0}[q,t]$. Combining this with the integrality statement yields the following result.

\begin{corollary}
For integers $k\geq 2$, $n\geq 1$, and $r\geq 1$ with $k\nmid n$, we have
\begin{equation}
\langle (-1)^{n-1}\nabla^{r}G(k,n), s_{\lambda}\rangle\in \N[q,t].
\end{equation}
\end{corollary}

\begin{remark}
If $n=ak$, the expansion in \eqref{equ-Gkn-C-m} will include a summand containing the operator $\C_{0}$. However, the expression $\nabla\C_0 \C_\alpha$ does not, in general, preserve Schur positivity for a composition $\alpha$.
\end{remark}

\bigskip
\noindent
\textbf{Acknowledgements:} We would like to thank Michele D'Adderio for valuable conversations.

\bibliographystyle{amsalpha}
\bibliography{Biblebib}

\end{document}